\documentclass[12pt]{amsart}
\usepackage{graphicx}
\usepackage{amssymb}
\usepackage{epstopdf}
\usepackage{nicefrac}
\usepackage{enumerate}
\usepackage{float}
\usepackage{hyperref}
\usepackage{tikz}
\usetikzlibrary{arrows}

\usepackage[all,cmtip]{xy}

\newtheorem{thm}{THEOREM}[section]

\newtheorem{cor}[thm]{COROLLARY}
\newtheorem{defn}[thm]{DEFINITION}
\newtheorem{ex}[thm]{EXAMPLE}

\newtheorem{lemma}[thm]{LEMMA}

\newtheorem{prop}[thm]{PROPOSITION}

\newtheorem{remark}[thm]{REMARK}

\newcommand{\e}{{\epsilon}}
\newcommand{\G}{{\Gamma}}

\newcommand{\mN}{{\mathbb N}}
\newcommand{\mQ}{{\mathbb Q}}
\newcommand{\mR}{{\mathbb R}}

\newcommand{\mX}{{\mathbb X}}
\newcommand{\mZ}{{\mathbb Z}}

\newcommand{\cC}{{\mathcal C}}

\newcommand{\cM}{{\mathcal M}}
\newcommand{\cN}{{\mathcal N}}

\newcommand{\cT}{{\mathcal T}}

\newcommand{\fM}{{\mathfrak{M}}}

\newcommand{\fX}{{\mathfrak{X}}}

\input xy
\xyoption {all}

\newcommand{\phom}{\mbox{{\rm Phom}}}
\begin{document}

\title{The homology core of matchbox manifolds and invariant measures}



\author{Alex Clark}
\thanks{IN-2013-045, Leverhulme Trust for International Network}
\address{Alex Clark, Department of Mathematics, University of Leicester, University Road, Leicester LE1 7RH, United Kingdom}
\email{Alex.Clark@le.ac.uk}

\author{John  Hunton}
\address{John  Hunton, Department of Mathematical Sciences, Durham University, South Road, Durham DH1 3LE, United Kingdom}
\email{john.hunton@durham.ac.uk}


\date{}

\maketitle
\begin{abstract}
 Here we consider the topology and dynamics associated to a wide class of matchbox manifolds, including spaces of aperiodic tilings and all minimal matchbox manifolds of dimension one. For such a space we introduce a topological invariant, the homology core, built using an expansion of it as an inverse sequence of simplicial complexes. The invariant takes the form of a monoid equipped with a representation, which in many cases can be used to obtain a finer classification than is possible with the previously developed invariants. When the space is obtained by suspending a topologically transitive action of the fundamental group $\G$ of a closed orientable manifold on a zero--dimensional compact space $Z$, this invariant corresponds to the space of finite Borel measures on $Z$ which are invariant under the action of $\G$. This leads to connections between the rank of the core and the number of invariant, ergodic Borel probability measures for such actions. We illustrate with several examples how these invariants can be calculated and used for topological classification and how it leads to an understanding of the invariant measures.
\end{abstract}

\section{Introduction}
Given the action of the fundamental group $\Gamma=\pi_1(M)$ of a closed orientable manifold $M$ on the zero--dimensional compact metric space $Z,$ one can suspend the $\G$ action over $M$ to form a space $\cM$. Provided this $\G$ action has a dense orbit (i.e., is topologically transitive), this space will have the structure of a {\em matchbox manifold}. More generally, a matchbox manifold is a compact, connected metric space that locally has the structure of $\mR^d\times \cT,$ where $\cT$ is a zero--dimensional space.  Such spaces occur naturally when considering minimal sets of foliations and hyperbolic attractors of diffeomorphisms of manifolds. We do not in general here require any differentiable structure, but shall restrict our consideration to matchbox manifolds that admit an expansion as an inverse system (tower) of finite simplicial complexes with well-behaved projection and bonding maps; such expansions are known to exist for many classes, see for example, \cite{CHL2014}. We detail the precise class of spaces we consider in Section \ref{back}.

Special and well studied examples of such objects include the so called tiling spaces arising from aperiodic tilings of a Euclidean space with finite (translational) local complexity; see Sadun's text \cite{S} for a general introduction to such examples. These tiling spaces can be viewed as the suspension over a torus of a $\mZ^d$ action on a zero--dimensional space, as detailed by Sadun \& Williams \cite{SW2003}.

In this paper we introduce a homeomorphism invariant of oriented matchbox manifolds, the {\em homology core\/} of $\cM$. In its strongest form, the core may be considered as a monoid equipped with a representation in a linear space, and it is this representation (as opposed to an abstract monoid) that distinguishes the homology core from earlier invariants constructed from the ordered cohomology. Our construction uses the top dimensional homology groups of an expansion of $\cM$, and as such the invariance of the homology core can be considered a generalisation of the work of Barge \& Diamond \cite{BDFund} and Swanson \& Volkmer \cite{SV2000} on the weak equivalence of matrices related to a one dimensional substitution tiling system; we have however no reason to restrict attention to substitution tiling spaces, and we can often compute our invariant for oriented matchbox manifolds of any dimension to great advantage.

Our work may also be seen as providing a generalisation of invariants as introduced by Kellendonk \cite{K} and Ormes, Radin \& Sadun \cite{ORS} which use oriented dimension or cohomology groups, applied to higher dimensional substitution tiling spaces; our use of homology however gives, even in the substitution tiling case, a richer invariant, yielding information distinguishing finer classifications than the earlier work by including the representation as part of the invariant. By using the inverse sequences that are dual to the sequences considered in ordered cohomology, many of the calculations are transparent that are not in the original setting. In particular, the homology core topologically distinguishes many spaces that have isomorphic cohomology and  captures some information related to the relative frequencies of cycles.

While the homology core is constructed using expansions, as is usual in shape theory, the homology core is not a shape invariant. In fact, we show that (unlike Cech cohomology) the homology core can be used to distinguish  examples of shape equivalent spaces, but at the same time, there are examples of spaces the homology core does not distinguish but which are distinguished by the authors' shape invariant $L_1$  defined in \cite{ClarkHunton}.

An intriguing and significant feature of the homology core we present appears when our underlying matchbox manifold $\cM$ is constructed by suspending a topologically transitive $\G$ action on a zero--dimensional compact space $Z$ over the oriented manifold $M$. We show under these circumstances that the top Cech cohomology of such a matchbox manifold is tractable, and as a result in many natural cases the homology core can be identified with the space of finite Borel measures on the space $Z$ that are invariant under the given action of $\G$. A related result for tiling spaces is  given by Bellissard, Benedetti \& Gambaudo \cite{BBG2006}.

There is a connection with objects that have been previously used in the study of invariant measures in, for example, Bezuglyi, Kwiatkowski, Medynets, \& Solomyak \cite{BKMS2013}, Aliste-Prieto \& D. Coronel \cite{APC2011}, Petite \cite{Pet} and Frank \& Sadun \cite{FS}. In those constructions the number of invariant, ergodic, Borel probability measures is usually found to be bounded above by the number of tile types, vertices in a related Bratteli diagram or similar information.  From our viewpoint, in many cases we can directly compute the number of invariant, ergodic Borel measures in terms of extreme points in our potentially much smaller homology core of $\cM$.

Furthermore, our result can be viewed as a refinement of the connection between the foliation cycles of a foliated space and the space of invariant measures discovered by Sullivan \cite{Sull1976}, see also \cite{MS2006}. The advantage of our approach is that one can calculate the homology core in a direct and quite tangible way, capturing some of the geometric information lost in the other approaches.

A novel feature of our approach is that it is purely topological and makes no use of a smooth structure as in \cite{Sull1976}, \cite{BBG2006}. At the same time, by considering the more general case we identify the key ingredients in the structure that make the argument go through. In particular, the fibred simplicial presentations as found in \cite{CHL2014} are essential in the arguments of Theorem~\ref{invariantsandcone}. Interestingly, these presentations are only known to exist in the case that $\G$ acts in a special way: if $\gamma \in \G$ fixes an element of the transverse space $z\in Z,$ then $\gamma$ must fix a neighbourhood of $z,$ thus highlighting some properties of the group action that appear necessary for the connection between the $\G$--invariant measures and the topologically invariant homology core. Since there are examples of group actions that do not admit non--trivial invariant Borel measures, this highlights a potential topological obstruction to the existence of such measures.

The paper is arranged as follows. In Sections 2 and 3 we specify the category of matchbox manifolds we consider, associated homology classes and their behaviour under homeomorphisms. In Section 4 we define the homology core, Definition \ref{core-def}, and prove its invariance under homeomorphism. We also introduce the properties of $\mZ$ and $\mQ$-stability. In Section 5 we concentrate on those matchbox manifolds that are suspensions over manifolds, and relate the homology core to spaces of invariant measures. In the final section, 6, we detail a number of examples and computations. We recover and generalise, Theorem \ref{Solenoid}, a result of Cortez \& Petite \cite{CortPet} on the unique ergodicity of certain solenoids and their associated odometers, provide examples, \ref{gencf}, for which the homology core can be used to identify a natural class of spaces that are not uniquely ergodic but for which the homology core can be calculated, examples, \ref{sub}, in which the core distinguishes between spaces with the same ordered cohomology, examples, \ref{cf}, which demonstrate that the core can distinguish shape equivalent spaces, but also, \ref{shape}, that the core will not fully distinguish between all shape inequivalent spaces.

\section{Background}\label{back}

In this section we shall present  the preliminary results that allow us to obtain the topological invariance of the homology core we construct in Section \ref{core}. We begin by recalling the suspension construction.

Let $\Gamma=\pi_1(M,m_0)$ be the fundamental group of a  $\mathrm{PL}$ closed orientable manifold $M$ of dimension $d.$ Let $\G$ act on the left of the zero--dimensional compact space $Z.$ We identify $\G$ with the deck transformations of the universal covering map $\widetilde{M} \to M $, and we consider $\G$ to act on the right of $\widetilde{M}.$ This then leads to the \emph{suspension} $\widetilde{M}\times_{\G}Z$, which is the orbit space of the action of $\G$  on  $\widetilde{M}\times Z$ given by

\[
(\gamma, (m,c) )\mapsto (m \cdot \gamma^{-1}, \gamma\cdot c).
\]

The space $\cM : = \widetilde{M}\times_{\pi_1(M)}Z$ thus constructed is a foliated space which is locally homeomorphic to $\mR^d\times Z.$ Provided that the action of $\G$ is topologically transitive, $\cM$ is connected and thus an example of a matchbox manifold.

\begin{defn} \label{def-mm}
A \emph{matchbox manifold}  $\cM$  is a compact, connected metric space with the structure of a
foliated space, such that for each $x \in \cM$, the
transverse model space $\cT_x$ is totally disconnected.
\end{defn}

The topological dimension of a matchbox manifold of dimension $d$ is the same as the dimension of its leaves, which coincide with the path components. In the case of a suspension over a manifold $M,$  $d$ coincides with the dimension of $M.$ The smoothness of a suspension $\widetilde{M}\times_{\pi_1(M)}Z$ along leaves in the case that $M$ is smooth and its structure as a fiber bundle over $M$ with fiber $Z$ follow from general considerations, see \cite[Chapt 3.1]{CandelConlon2000}. A matchbox manifold is \textit{minimal} when each path component is dense. A suspension $\widetilde{M}\times_{\pi_1(M)}Z$ is minimal if the action of $\G$ on $Z$ is minimal. We refer the reader to \cite{ClarkHurder2010}, \cite{CHL2014} for a more detailed discussion.

\begin{defn}
Let  $\cM$ be a matchbox manifold of dimension $d$. A \emph{simplicial presentation} of $\cM$ is an inverse sequence whose limit is homeomorphic to $\cM$
 \[
 \begin{array}{lcl}
 \cM&\approx & \underleftarrow{\lim}\{M \xleftarrow{f_1}X_2 \xleftarrow{f_2} X_3 \xleftarrow{f_3} \cdots\}
\end{array}
 \]
and is such that each $X_n$ is a triangulated space and each bonding map  $f_n$ is surjective and simplicial. Additionally, we require for each $n$ that:
\begin{itemize}
 \item[(i)] each simplex in the triangulation of $X_n$ is a face of a $d$--dimensional simplex and
 \item [(ii)] each  $d$--dimensional simplex $S$ of $X_n$ pulls back in $\cM$ to a subset homeomorphic to $S\times K$ for some zero--dimensional compact $K$ and that for each $k\in K$ the restriction of the projection $\cM \to X_n$ to $S\times \{k\}$ is a homeomorphism onto its image.
 \end{itemize}
\end{defn}

Condition (ii) is similar to requiring the restrictions to leaves to be covering maps (as is the case of the fiber bundle projections in a suspension), only at the boundaries of the simplices $S$ in the leaves (where there can be branching in $X_n$) the projections do not necessarily behave as covering maps. In addition to very general tiling spaces, according to the results of \cite{CHL2014}, a wide variety of minimal matchbox manifolds admit such a presentation.

According to the definition, corresponding to the triangulation of $X_n$ in a simplicial presentation of $\cM$ admits, there is  a decomposition of $\cM$ into a finite number of sets of the form $S_i \times K_i$ that intersect only along sets of the form $\partial S_i \times K,$ where $\partial S_i$ is the boundary of $S_i$ and $K$ is a clopen subset of $K_i.$ Thus, the leaves of $\cM$ can be given a simplicial structure induced by this decomposition.  What is more, the leaves of $\cM$ can be considered as being tiled by finitely many tile types, one type corresponding to each simplex $S_i$ in the triangulation of $X_n.$ Given the nature of a triangulation, we also have that there are only finitely many ways that tiles may intersect in a leaf, which can be considered as a form of what is known as finite local complexity. Each of the successive approximating spaces $X_n$ leads to a finer decomposition of $\cM$ and the fibers of the projection $\cM \to X_{n+1}$ are contained in the fibers of the projection $\cM \to X_{n}$ and the induced map $f_n:X_{n+1}\to X_n$ is simplicial in that it can be considered as the geometric realisation of a simplicial map of the complexes underlying the triangulations of $X_{n+1}$ and $X_n.$

This special structure will allow us to apply a powerful result on the approximation of maps between inverse limits as described below.

\begin{defn}
For given inverse limits $\cM = \displaystyle\,{\lim_{\longleftarrow}} \{X_n,f_n\}$ and $\cN = \displaystyle{\lim_{\longleftarrow}} \{Y_n,g_n\}$, a map  $h:\cM \to \cN$ is said to be \emph{induced} if for a subsequence $n_i$ of $\mN,$ there is for each $i\in \mN$  a map $h_i: X_{n_i} \to Y_i$ such that the following diagram commutes
\[
\xymatrix{
X_{n_1} \ar[d]_{h_1}
& X_{n_2} \ar[l]_{f^{n_2}_{n_1}}\ar[d]^{h_2}
& \cdots\ar[l]
& X_{n_k}\ar[d]^{h_k}\ar[l]
& X_{n_{k+1}}\ar[l]_{f^{n_{k+1}}_{n_k}}\ar[d]^{h_{k+1}}
& \cdots\ar[l]\\
Y_1
& Y_2\ar[l]^{g_1}
&\cdots\ar[l]
& Y_k\ar[l]
& Y_{k+1}\ar[l]^{g_k}
& \cdots\ar[l]}
\]
and the resulting map $\cM \to \cN$ given by $(x_i) \mapsto (h_i(x_{n_i}))$ is equal to $h.$
\end{defn}
In the above, for $k<\ell,$ $f^\ell_k=f_k\circ \cdots \circ f_{\ell -1}.$

We record the key result of Rogers \cite[Thm 4]{R1973} on the approximation of maps between inverse limits as maps between the factor spaces.
\begin{thm}\cite{R1973}\label{Rogers}
Given two matchbox manifolds  with simplicial presentations $\cM = \displaystyle\,{\lim_{\longleftarrow}} \{X_n,f_n\}$ and $\cN = \displaystyle{\lim_{\longleftarrow}} \{Y_n,g_n\}$ and given any $\e>0$, any continuous map $f:\cM\to \cN$ is homotopic to an induced map $f_\e$ in which points are moved no more than $\e$ over the course of the homotopy from $f$ to $f_\e.$
\end{thm}

\section{Orientation in matchbox manifolds}

We now consider a  matchbox manifold $\cM $  with simplicial presentation $\cM = \displaystyle\,{\lim_{\longleftarrow}} \{X_n,f_n\}$.  Recall that the leaf topology for a leaf $L$ has a basis of  open sets  formed by  intersecting $L$ with open sets of plaques of the foliation charts, which gives $L$ the structure of a connected manifold. A leaf can be orientable or not, and when $L$ is orientable it admits one of two orientations. (A convenient way of considering orientations and orientability for a non--compact manifold admitting a simplicial structure such as $L$ is with the use of homology groups based on infinite chains, see, e.g., \cite[p.33, 388]{Munk}.)  If $L$ is orientable, each time it enters a subset of $\cM$ of the form $S \times K,$ where $S$ is a simplex of dimension $d$ corresponding to a triangulation of some $X_n$, $L$ induces an orientation of $S.$ It can happen that each time an oriented $L$ enters $S \times K$ it induces the same orientation of $S$ or it could induce different orientations.  If $L$ always induces the same orientation on $S,$ we shall say  $L$ induces a \textit{coherent} orientation on $S.$ In a  minimal matchbox manifold $\cM$, whether a given simplex $S$ of $X_n$ is coherently oriented is independent of the choice of orientable leaf $L.$ (It should be borne in mind that for general matchbox manifolds not all leaves of a matchbox manifold need be homeomorphic and that it can even happen that some leaves are orientable while others not.)
\begin{defn}
A simplicial representation of a matchbox manifold $\cM = \displaystyle\,{\lim_{\longleftarrow}} \{X_n,f_n\}$ is \emph{orientable} if the following conditions hold
 \begin{itemize}
   \item[(i)] $\cM$ has an orientable dense leaf $L$ and
   \item[(ii)] $L$ can be oriented coherently with all the simplices occurring in the triangulations of the $X_n.$
\end{itemize}
An \emph{orientation} of an orientable simplicial presentation of $\cM$ is given by a choice of orientation of a dense leaf $L$ as above and the corresponding induced orientation of each simplex occurring in the triangulations of the $X_n.$
\end{defn}
From here we shall only consider orientable presentations. While this originally seems quite restrictive, any matchbox manifold has a an orientable double ``cover", \cite[p. 280]{CandelConlon2000}. Also, the leaves of a tiling space arising from an aperiodic tiling of $\mR^d$ with finite translational local complexity  admit a natural orientation induced by the translation action, and the various presentations that have been constructed using the structure of the tiles are coherent with this orientation provided one takes the extra step of introducing the simplicial structure on the complexes. Observe also that since we are endowing each $X_n$ with the orientation induced by $L$ and the bonding maps are simplicial, the bonding maps will preserve the orientation of each simplex.

\begin{defn}
A homeomorphism $h: \cM \to \cN$ of matchbox manifolds with corresponding oriented simplicial presentations $\cM = \displaystyle\,{\lim_{\longleftarrow}} \{X_n,f_n\}$ and $\cN = \displaystyle{\lim_{\longleftarrow}} \{Y_n,g_n\}$ with orientations induced by the leaf $L$ of $\cM$ and $h(L)$ of $\cN$   is  \emph{orientation preserving} if $h$ preserves the orientation of $L$ and otherwise $h$ is \emph{orientation reversing}.
\end{defn}

The invariants we construct will be preserved by orientation preserving maps and are intimately related to how their homotopic induced maps act on the algebraic invariants of the approximating spaces $X_n.$

\begin{defn}
Given an oriented simplicial presentation $\cM = \displaystyle\,{\lim_{\longleftarrow}} \{X_n,f_n\},$  a \emph{positive homology class} of $X_n$ is a homology class in the simplicial homology $H_d(X_n)$ that can be represented as the positive integer combination of elementary chains of positively oriented $d$--simplices of some simplicial subdivision of $X_n$, and we denote the set of all positive homology classes as $H_d^+(X_n).$ Similarly, we define $H_d^-(X_n)$ as all the homology classes in the simplicial homology $H_d(X_n)$ that can be represented as the negative integer combination of elementary chains of positively oriented $d$--simplices of some simplicial subdivision of $X_n.$ (The zero class is considered to be in $H_d^+(X_n)\cap H_d^-(X_n)$.)
\end{defn}
Observe that by our choices of orientation for the $X_n$ and their common relation to a chosen leaf $L$, each bonding map $f_n:X_{n+1} \to X_n$ satisfies $f_n(H_d^+(X_{n+1}))\subset  H_d^+(X_n)$ and similarly with $H_d^-(X_{n+1}).$
The following result is key for the topological invariance of the homology core.
\begin{prop}\label{coneprop}
Given a homeomorphism $h: \cM \to \cN$ of $d$--dimensional matchbox manifolds with corresponding oriented simplicial presentations $\cM = \displaystyle\,{\lim_{\longleftarrow}} \{X_n,f_n\}$ and $\cN = \displaystyle{\lim_{\longleftarrow}} \{Y_n,g_n\}$ let   $h':\cM \to \cN$ be any induced homotopic map corresponding to the following commutative diagram
\[
\xymatrix{
X_{n_1} \ar[d]_{h_1}
& X_{n_2} \ar[l]_{f^{n_2}_{n_1}}\ar[d]^{h_2}
& \cdots\ar[l]
& X_{n_k}\ar[d]^{h_k}\ar[l]
& X_{n_{k+1}}\ar[l]_{f^{n_{k+1}}_{n_k}}\ar[d]^{h_{k+1}}
& \cdots\ar[l]\\
Y_1
& Y_2\ar[l]^{g_1}
&\cdots\ar[l]
& Y_k\ar[l]
& Y_{k+1}\ar[l]^{g_k}
& \cdots\ar[l]}
\]
Then for each $i\in \mN,$ either

\begin{itemize}
  \item[(i)] $(h_i)_*(H_d^+(X_{n_i}))\subset H_d^+(Y_i)$ and $(h_i)_*(H_d^-(X_{n_i}))\subset H_d^-(Y_i)$ or
  \item[(ii)]  $(h_i)_*(H_d^+(X_{n_i}))\subset H_d^-(Y_i)$ and $(h_i)_*(H_d^-(X_{n_i}))\subset H_d^+(Y_i)$
\end{itemize}
according as $h$ is orientation (i) preserving or (ii) reversing.
\end{prop}
\begin{proof}
Suppose then that we have an orientation preserving homeomorphism $h: \cM \to \cN$ with homotopic induced map $h':\cM \to \cN$ as above, and let $i\in \mN.$  To calculate the map induced on homology  $(h_i)_*:H_d^+(X_{n_i})\to H_d^+(Y_i),$ one first finds a simplicial approximation $H: X_{n_i} \to Y_i$ to $h_i.$ Notice that this simplicial approximation also induces a simplicial map $H_L: L \to h'(L).$ As the path components coincide with the leaves of these spaces and $h'$ is homotopic to $h,$ we have $h'(L)=h(L).$ By hypothesis, $h$ preserves the orientation and maps the positive generator of $H_d(L)$, which is the class formed by the sum of all the elementary chains of positively oriented simplices of dimension $d$, to the positive generator of $H_d(h(L)).$ The same is true then for the homotopic map $h'$ (and the map it induces on leaves) and so also for the simplicial approximation $H_L.$ But that means that $H_L$ must map positively oriented simplices to positively oriented simplices or degenerate simplices. The other cases are similar.
\end{proof}

It is important to realise that even when the underlying map $h$ is an induced homeomorphism, the maps $h_n$ are often not homeomorphisms.

\section{Homology core}\label{core}
In this section we shall introduce the homology core and show the subtle ways it is preserved by homeomorphism, depending on the precise nature of the space in question. In what follows we consider an oriented simplicial presentation $\cM = \displaystyle\,{\lim_{\longleftarrow}} \{X_n,f_n\}$ of a matchbox manifold. We first observe that the groups $H_d(X_{n})$ are free abelian of some finite rank. Now consider the subgroup $P_n$ of $H_d(X_{n})$ generated by $H_d^+(X_n),$ which will then also be a free group of rank say $r_n.$ Let now $V_n:= P_n\bigotimes \mR,$ an $\mR$--vector space of dimension $r_n.$ As previously observed, $(f_n)_*$ maps $H_d^+(X_{n+1})$ into $H_d^+(X_n),$ and so $(f_n)_*\bigotimes id_\mR$ yields a linear map $L_n:V_{n+1}\to V_n.$ (When needing to distinguish these vector spaces or maps for different spaces we add a superscript, e.g. $L_n^\cM.$)

\begin{defn}
The \emph{positive and negative cone} in $V_n$ is
\[
\cC_n:=\left\{\sum  x_i\bigotimes r_i \,|\, r_i \geq 0,\, x_i \in H_d^+(X_n)\,\right\} \cup  \left\{\sum  x_i\bigotimes r_i \,|\, r_i \leq 0,\, x_i \in H_d^+(X_n)\,\right\}
\]
and the \emph{positive cone} in $V_n$ is
\[
\cC^+_n:=\left\{\sum  x_i\bigotimes r_i \,|\, r_i \geq 0,\, x_i \in H_d^+(X_n)\,\right\}.
\]
\end{defn}

Our previous observations can then be rephrased has $L_n(\cC_{n+1}) \subset \cC_n.$ However, this inclusion will often be strict. This leads us to the following.

\begin{defn}\label{core-def}
We define the \emph{homology core at place} $k$ of the oriented presentation $\cM = \displaystyle\,{\lim_{\longleftarrow}} \{X_n,f_n\}$ and linear maps $L_n:V_{n+1}\to V_n$ as above by
\[
\cC_{\cM}(k):= \bigcap_{n>k} \,L_k^n \left( \cC_n \right),
\]
where $L_k^n= L_k\circ  \cdots \circ L_{n-2}\circ L_{n-1}.$ Similarly, the {\em positive homology core} $C^+_\cM(n)$ is given by
\[
\cC^+_{\cM}(k):= \bigcap_{n>k} \,L_k^n \left( \cC^+_n \right).
\]
\end{defn}

The fact that we must consider the core at various ``places" $k$ is a reflection of the fact that induced maps of towers homotopic to a given map do not have to respect the places in the two corresponding towers.

\begin{thm}\label{topinv}
Suppose we have a homeomorphism $h: \cM \to \cN$ of $d$--dimensional matchbox manifolds with corresponding oriented simplicial presentations $\cM = \displaystyle\,{\lim_{\longleftarrow}} \{X_n,f_n\}$ and $\cN = \displaystyle{\lim_{\longleftarrow}} \{Y_n,g_n\}.$

\begin{itemize}
  \item[(i)] Then there are subsequences $m_i,n_i$ and linear maps $K_i$ with $J_i$ that map the cones in the following diagram surjectively.

\begin{equation}\label{conediag}
\xymatrix{
\cC_{\cM}(n_1) \ar@{->>}[d]^{K_1}
& \cC_{\cM}(n_2) \ar@{->>}[l]\ar@{->>}[d]^{K_2}
& \cC_{\cM}(n_3) \ar@{->>}[l]\ar@{->>}[d]^{K_3}
& \cdots\ar@{->>}[l]
& \cC_{\cM}(n_k)\ar@{->>}[d]^{K_k}\ar@{->>}[l]
& \cC_{\cM}(n_{k+1)}\ar@{->>}[l]\ar@{->>}[d]^{K_{k+1}}
& \cdots\ar@{->>}[l]\\
\cC_{\cN}(m_1)
& \cC_{\cN}(m_2)\ar@{->>}[l]\ar@{->>}[lu]_{J_1}
& \cC_{\cN}(m_3)\ar@{->>}[l]\ar@{->>}[lu]_{J_2}
&\cdots\ar@{->>}[l]
& \cC_{\cN}(m_k)\ar@{->>}[l]
& \cC_{\cN}(m_{k+1})\ar@{->>}[l]\ar@{->>}[lu]_{J_k}
& \cdots\ar@{->>}[l]}
\end{equation}
  \item[(ii)] If the linear maps $L^\cN_n$ are eventually injective, then there is an  $\ell$ such that each linear map $J_i$ ($i\geq \ell$) as in the above Diagram \ref{conediag} maps $\cC_{\cN}(m_{i+1})$ isomorphically onto $\cC_{\cM}(n_i).$
  \item[(iii)] Moreover, if there is a uniform (for all $n$) bound to $\dim V_n^\cN,$ then there is an  $\ell$ such that each linear map $J_i$ ($i\geq \ell$) as in the above Diagram \ref{conediag} maps $\cC_{\cN}(m_{i+1})$ isomorphically onto $\cC_{\cM}(n_i).$
\end{itemize}

\end{thm}

\begin{proof}
(i) Suppose then that  $h',h''$ are induced maps as in Theorem \ref{Rogers} corresponding to $h$ and  $h^{-1}.$ This then leads to the following diagram between subtowers after reindexing:

\[
\xymatrix{
X_{n_1} \ar[d]^{h_1}
& X_{n_2} \ar[l]_{f^{n_2}_{n_1}}\ar[d]^{h_2}
& X_{n_3} \ar[l]_{f^{n_3}_{n_2}}\ar[d]^{h_2}
& \cdots\ar[l]
& X_{n_k}\ar[d]^{h_k}\ar[l]
& X_{n_{k+1}}\ar[l]_{f^{n_{k+1}}_{n_k}}\ar[d]^{h_{k+1}}
& \cdots\ar[l]\\
Y_{m_1}
& Y_{m_2}\ar[l]^{g_{m_1}^{m_2}}\ar[lu]_{\ell_1}
& Y_{m_3}\ar[l]^{g_{m_2}^{m_3}}\ar[lu]_{\ell_2}
&\cdots\ar[l]
& Y_{m_k}\ar[l]
& Y_{m_{k+1}}\ar[l]^{g_{m_{k}}^{m_{k+1}}}\ar[lu]_{\ell_k}
& \cdots\ar[l]}
\]
In general this diagram will \textit{not} be commutative, but the maps induced on homology  are commutative since the compositions of $h'$ and $h''$ are homotopic to the respective identities. By Proposition \ref{coneprop} we are then led to the following commutative diagram of (restrictions of) linear maps.

\begin{equation}\label{conediag2}
\xymatrix{
\cC_{\cM}(n_1) \ar[d]
& \cC_{\cM}(n_2) \ar[l]\ar[d]
& \cC_{\cM}(n_3) \ar[l]\ar[d]
& \cdots\ar[l]
& \cC_{\cM}(n_k)\ar[d]\ar[l]
& \cC_{\cM}(n_{k+1)}\ar[l]\ar[d]
& \cdots\ar[l]\\
\cC_{\cN}(m_1)
& \cC_{\cN}(m_2)\ar[l]\ar[lu]
& \cC_{\cN}(m_3)\ar[l]\ar[lu]
&\cdots\ar[l]
& \cC_{\cN}(m_k)\ar[l]
& \cC_{\cN}(m_{k+1})\ar[l]\ar[lu]
& \cdots\ar[l]}
\end{equation}

By construction, each horizontal map in the diagram is surjective. Each of the  maps $K_i$ is part of a commutative triangle:

\begin{equation}\label{triangle}
\xymatrix{
\cC_{\cM}(n_i) \ar[d]^{K_i}
&  \\
\cC_{\cN}(m_i)
& \cC_{\cN}(m_{i+1})\ar@{->>}[l] \ar[lu]_{J_i}
}.
\end{equation}
As the horizontal map is surjective, $K_i$ is surjective as required in Diagram \ref{conediag}. Similar arguments apply to the $J_i.$

(ii) Assume then that the horizontal maps are additionally injective from some point in the  tower associated to $\cN.$ Then for sufficiently large $i$ in the triangle \ref{triangle}, we see  $J_i$ must also be injective on $\cC(m_{i+1}).$

(iii) Assume then that there is a uniform (for all $n$) bound to $\dim V_n^\cN$. This then implies a uniform bound for the \textit{topological} dimension of $\cC_{\cN}(m).$ As the maps  $(L_n^m)^\cN$ are linear, they cannot raise (topological) dimension. Hence, for sufficiently large values (say, $m\geq N$) the topological dimension of $\cC_{\cN}(m)$ must have the same value $D.$ Thus, for all $k>N$ the restriction of the maps $(L_N^k)^\cN$ to $\cC_{\cN}(k)$ must be injective and we have the hypothesis for (ii).
\end{proof}

The spaces which are best understood are those for which there is a uniform bound on $\textrm{rank} \,H_d(X_n)$ for a presentation. In such cases, by telescoping the given presentation, one can obtain a presentation for which $\textrm{rank} H_d(X_n)$ is constant. We can already see from the above the homology core yields a good deal of information for such spaces. However, depending on the exact conditions we can say much more in special cases.

\begin{defn}\label{stable}
An oriented simplicial presentation $\cM = \displaystyle\,{\lim_{\longleftarrow}} \{X_n,f_n\}$   is said to be \emph{homologically $\mZ$--stable} if  for each $n$ $(f_n)_*  : H_d(X_{n+1})\to H_d(X_n)$ and   $L_n:V_{n+1}\to V_n$ are isomorphisms, and we say the presentation is \emph{homologically $\mQ$--stable} if  for each $n$ $(f_n)_*  : H_d(X_{n+1},\mQ)\to H_d(X_n,\mQ)$ and   $L_n:V_{n+1}\to V_n$ are isomorphisms.
\end{defn}

\begin{remark}\label{invlim}
\em{Observe that in the case of a $\mQ$--stable presentation, each core $\cC_{\cM}(k)$ can be identified with the limit of the inverse sequence of the cones $\cC_n$ $(n\geq k)$ with bonding maps the restrictions of the $L_n$. The core $\cC_{\cM}(k)$, however, retains some of the geometric information lost in the limit as it includes an embedding in $V_k.$  However, this identification with the inverse limit will be significant when relating the cores to their dual counterparts in the following section.}
\end{remark}

\begin{thm}\label{orbit}
For $\mX \in \{\,\mZ,\mQ\,\},$  suppose we have homologically $\mX$--stable simplicial presentations $\cM = \displaystyle\,{\lim_{\longleftarrow}} \{X_n,f_n\}$ and $\cN = \displaystyle{\lim_{\longleftarrow}} \{Y_n,g_n\}$ and a homeomorphism $h: \cM \to \cN.$ Choose a basis for each $V_n^\cM,V_n^\cN$ consisting of elements of $H_d^+(X_n),\,H_d^+(Y_n)$ so that the corresponding linear maps $L^\cM_n,\, L^\cN_n$ are represented by elements of $GL(D,\mX)$, where $D$ is the common dimension of the $V_n^\cM,V_n^\cN.$ Then, with respect to these bases, all the homology cores $\cC_{\cM}(m)$ and $\cC_{\cN}(n)$ are in the same $GL(D,\mX)$--orbit.
\end{thm}

\begin{proof}
Consider now diagram \ref{conediag} as before and the associated diagram on homology groups with $\mX$ coefficients. Under our new hypotheses, all the horizontal maps are isomorphisms and hence all the vertical and diagonal maps are also isomorphisms as well. With the bases we have chosen, the result follows directly.
\end{proof}
Applying similar arguments to the general case as Theorem~\ref{topinv} one obtains that for homeomorphic matchbox manifolds $\cM,\cN$ (without stability requirements) that the cores $\cC_{\cM}(m)$ and $\cC_{\cN}(n)$ are images of matrices with integer entries for \emph{restricted} choices of $m,n$ as indicated in the theorem. Observe that if there is a uniform bound on $\textrm{rank} \, H_d(X_n)$, then we can find an inverse sequence of groups which is $\mQ$--stable and which is pro--equivalent to the inverse sequence of the the $V_n$ and $L_n$, and this is sufficient to draw the same conclusions as in the above theorem.

\section{Homology core and invariant measures}

In this section we consider matchbox manifolds that are given as suspensions $\cM : = \widetilde{M}\times_{\G}Z$  over a closed $\mathrm{PL}$ manifold $M$ with fundamental group $\G=\pi_1(M).$ We denote the associated bundle projection  by $\pi: \cM \to M.$ In this setting, in many cases we can directly relate the homology core to the space of $\G$--invariant, finite Borel measures fora a topologically transitive $\G$--actions on the zero dimensional space $Z.$ We shall only finite measures and henceforth will assume all measures are finite.

\begin{defn}
Let us say that the matchbox manifold $\cM=\widetilde{M}\times_{\G}Z$ has a {\em consistent presentation\/} if it has an oriented simplicial presentation $\cM= \displaystyle\,{\lim_{\longleftarrow}} \{X_n,f_n\}$ for which the fibers of the projection $\cM \to X_1$ are subsets of the fibers of the bundle map $\pi: \cM \to M.$
\end{defn}

The class of matchbox manifolds obtained by such a suspension construction is quite large and includes all translational tiling spaces of finite local complexity \cite{SW2003}. However, it does not include all two-dimensional orientable examples, \cite{FarrellJones1981}.

\begin{defn} Denote by $\fM(Z)$ the set of all Borel measures on the space $Z$. For a ring $R=\mZ$ or $\mR$, denote by $C(Z;R)$ the $R$-module of continuous $R$-valued functions on $Z$. A {\em positive element\/} of $C(Z;R)$ is a function that takes only non-negative values. A {\em positive homomorphism\/} $C(Z;R)\to \mR$ is an $R$-linear map which takes positive elements to non-negative numbers. We write $\phom_R(C(Z;R);\mR)$ for the set of positive homomorphisms $C(Z;R)\to \mR$.
\end{defn}

\begin{lemma}\label{PHOM}
$$\fM(Z)=\phom_\mZ(C(Z;\mZ);\mR)\,.$$
\end{lemma}

\begin{proof}
The Riesz Representation theorem tells us that the set of measures $\fM(Z)$ can be identified with $\phom_\mR(C(Z;\mR);\mR)$, where $\mu\in \fM(Z)$ corresponds to the positive homomorphism $f\mapsto\int_Zfd\mu$. Any functional $\int_Z-\,d\mu$ is however determined by its values on $\mR$-valued step functions taking finitely many values; this set of functions can be equated with $C(Z;\mZ)\otimes\mR$. The lemma follows by noting the equivalence
$$\phom_\mR(C(Z;\mZ)\otimes\mR;\mR)\equiv\phom_\mZ(C(Z;\mZ);\mR)\,.$$
\end{proof}

\begin{prop}\label{cohMM}
Suppose the $d$-dimensional matchbox manifold $\cM=\widetilde{M} \times_{\G}Z$ has a consistent presentation. Then the top dimension Cech cohomology, $H^d(\cM)$ can be identified with $C(Z;\mZ)_\G$, the $\G$-coinvariants of $C(Z;\mZ)$.
\end{prop}
\begin{proof}
A Serre spectral sequence for the cohomology of $\cM$ using the bundle structure
$$Z\longrightarrow \widetilde{M}\times_{\G}Z=\cM \xrightarrow{\pi} M$$
yields an $E_2$ page
$$E_2^{p,q}=\left\{
\begin{array}{ll}
H^p(M;H^0(Z))=H^p(M;C(Z;\mZ))&\mbox{ if }q=0\\
0&\mbox{ if }q\not=0\,.
\end{array}\right.$$
This follows from the fact that the Cech cohomology of a totally disconnected space $Z$ is $C(Z;\mZ)$ in dimension 0 and is trivial in all higher dimensions.
The spectral sequence thus collapses, with no extension problems, giving $H^p(\cM)=H^p(M;C(Z;\mZ))$. To conclude the proof, we show that in general, a group cohomology $H^d(M;A)$ for a closed, orientable triangulated $d$-manifold with fundamental group $\G$ and coefficients $A$ with (potentially non-trivial) $\G$-action can be identified with the coinvariants $A_\G$.

Lift the triangulation of $M$ to a triangulation on the universal cover $\widetilde{M}$, and consider $C^d_\G(\widetilde{M};A)$, the $\G$-equivariant $d$-cochains on $\widetilde{M}$ with values in $A$. As $M$ is compact these form a free, finite dimensional $A$-module. As we can find a path from the interior of one $d$-simplex to that of any other, passing only through $d-1$ simplices, the cohomology
\begin{equation}\label{cohomology}
H^d(M;A)=\frac{C_\G^d(\widetilde{M};A)}{\mbox{Im}\left(\delta^d\colon C_\G^{d-1}(\widetilde{M};A)\to C_\G^d(\widetilde{M};A)\right)}
\end{equation}
is generated by a single copy of $A$. However, for each $\gamma\in\G$ and each $d$-simplex $\sigma$, there is a path from the interior of $\sigma$ to itself which represents $\gamma$, and crosses only codimension one simplices. The sum of the coboundaries of these $d-1$ simplices, taken over all $\gamma\in\G$, show that the quotient (\ref{cohomology}) is the full coinvariants $A_\G$.
\end{proof}

\begin{remark}{\em
In the situation where the manifold $M$ is also aspherical, we can prove more. This case includes any $d$-torus, as is the case when $\cM$ is a tiling space for a $d$-dimensional tiling of finite local complexity, and also the case when $M$ is any Riemannian manifold of non-positive curvature. If $M$ is aspherical (so, $\pi_n(M)=0$ for all $n>1$), then $M$ is a model for the classifying space $B\G$. The cohomology $H^p(M;C(Z;\mZ))$ can thus be identified with the group homology $H^p(\G;C(Z;\mZ))$. Moreover, the Poincar\'e duality of the manifold $M$ tells us that $\G$  is a Poincar\'e duality group, and this latter property implies that, for any $\G$-module $A$, the group homology and cohomology of $\G$ with coefficients in $A$ are related by the isomorphism
\begin{equation}H_n(\G;A)\cong H^{d-n}(\G;A)\,.\end{equation}
The conclusion of Proposition \ref{cohMM} now follows since
$$H^d(\cM)=H^d(M;C(Z;\mZ))=H_0(\G;C(Z;\mZ))=C(Z;\mZ)_\G$$
where the last equivalence can be taken as the definition of group homology (i.e., that for a given group $\G$, the group homologies $H_*(\G;-)$ are the left derived functors of the coinvariant functor $A\mapsto A_\G$; see, for example, \cite{KB} section II.3). Clearly though, for such manifolds $M$ more is true and the intermediate dimensional cohomology can be described in a fashion similar to that used in \cite{JRH} section 3.
}
\end{remark}

\begin{cor}\label{invmeas}
Suppose $\cM$ is an oriented matchbox manifold of dimension $d$ with a consistent presentation. Then $\fM^\G\!(Z)$, the $\G$-invariant measures on $Z$, can be identified
$$\fM^\G\!(Z)=\phom_\mZ(C(Z;\mZ)_\G;\mR)=\phom_\mZ(H^d(\cM);\mR)\,.$$
\end{cor}

\begin{proof}
As the $\G$ action on $Z$ induces an action on $C(Z;\mZ)$ which takes positive elements to positive elements, a simple adjunction yields the identification
$$\begin{array}{rl}
\fM^\G(Z)=&\!\!\!\mbox{Positive $\G$-invariant $\mZ$-linear homomorphisms }C(Z;\mZ)\to\mR,\quad\mbox{by Lemma }\ref{PHOM}\\
=&\!\!\!\phom_\mZ(C(Z;\mZ)_\G;\mR)\\
=&\!\!\!\phom_\mZ(H^d(\cM);\mR),\qquad\mbox{by Lemma }\ref{cohMM}\,.
\end{array}$$
\end{proof}

The set of positive homomorphisms $\phom_\mZ(H^d(\cM);\mR)$, being dual to a cohomological gadget, has a natural homological interpretation, the homology core of Section \ref{core}.

\begin{thm}\label{invariantsandcone}
Let $\cM$ be an oriented matchbox manifold of dimension $d$ with a consistent presentation, and assume also that the presentation is homologically $\mQ$-stable and the $H^+_d(X_n)$ generate the whole of $H_d(X_n)$ for all sufficiently large $n$. Then for any $n$, the space of $\G$--invariant, Borel measures on $Z$ can be identified with the positive homology core
$$\fM^\G(Z)=C^+_\cM(n)\,.$$
\end{thm}
\begin{proof}
We  make identifications
$$\begin{array}{rll}
\fM^\G(Z)=&\phom_\mZ\left(C(Z,\mZ)_\G;\mR\right)&(1)\\&&\\
=&\phom_\mZ\left(\displaystyle\lim_{\rightarrow}H^d(X_n);\mR\right)&(2)\\&&\\
=&\displaystyle\lim_{\leftarrow}\phom_\mR\left(\hom\left(H_d(X_n);\mR\right);\mR\right)&(3)\\&&\\
=&\displaystyle\lim_{\leftarrow}H_d^+(X_n;\mR)&(4)\\&&\\
=&C^+_\cM(n)\,.&(5)
\end{array}$$
Here, (1) is a restatement of Corollary \ref{invmeas}, and (2) uses the property of Cech cohomology, that $H^d(\cM)=H^d(\displaystyle\lim_{\leftarrow}X_n)=\displaystyle\lim_{\rightarrow}H^d(X_n)$. The identification (3) uses the observation that as we are looking at linear maps to $\mR$ we might as well consider  homology and cohomology with coefficients in $\mR$, and in that case the homology and cohomology are dual vector spaces. The lines (2) and (3) also require an understanding of positivity in the cohomology groups $H^d(X_n)$, that the maps in the direct system $\displaystyle \lim_{\rightarrow}H^d(X_n)$ preserve positivity, and that the notion of positivity in the limit $\displaystyle \lim_{\rightarrow}H^d(X_n)$ agrees with that in $C(Z;\mZ)_\G$.

By the universal coefficient theorem, the free part of the cohomology $H^d(X_n)$ is given by the dual group $\hom(H_d(X_n);\mZ)$ (we need not worry about any torsion part as eventually all passes to $\mR$ coefficients). This gives a notion of positivity in this part of the cohomology, where we say $[\alpha]\in H^d(X_n)$ is positive if, as an element of $\hom(H_d(X_n);\mZ)$ it takes non-negative values on positive homology elements, i.e.,
$$[\alpha]\in H^d(X_n)\mbox{ is positive if }\alpha(x)\geq0\mbox{ for all }x\in H^+_d(X_n)\,.$$
As $(f_n)_*\colon H_d(X_{n+1})\to H_d(X_n)$ takes positive elements to positive elements, so too does $(f_n)^*\colon H^d(X_{n})\to H^d(X_{n+1})$ preserve positivity and we can define the positive part of $H^d(\cM)$ as the direct limit of the positive parts of the $H^d(X_n)$. We must show that this coincides with the notion of positivity in $C(Z;\mZ)_\G$ under the isomorphism of Proposition \ref{cohMM}.

Let $\sigma$ be a $d$-simplex in $M$, and pick an interior point $z$. Regard $Z$ as the fibre in $\cM$ over $z$, and let $Z_n$ be the finite, discrete space, the image of $Z$ in $X_n$. Let $\alpha$ be a continuous function $Z\to\mZ$, representing the class $[\alpha]\in C(Z;\mZ)_\G=H^d(\cM)$. The perspective of Proposition \ref{cohMM} shows that $\alpha$ can be interpreted as a cocycle on $\cM$ by first mapping $Z\times \sigma\to\mZ$ using $\alpha$ on $Z$ (and constant on the $\sigma$ component), extended trivially to the rest of $\cM$. Any such cocycle is the pull back of a cocycle on $X_n$,  for some sufficiently large $n$, defined similarly using some function $\alpha_n\colon Z_n\to\mZ$. Then $\alpha$ is a positive function if and only if $\alpha_n$ is. Clearly if $\alpha_n$ is a positive function, the cohomology class $[\alpha]\in H^d(X_n)$ is positive in the sense above, that as a homomorphism on $H_d(X_n)$ it takes non-negative values on $H^+_d(X_n)$. This shows that the positives in $C(Z;\mZ)_\G$ map to the positives in $H^d(\cM)$, moreover they map injectively since this is just the restriction of the isomorphism $C(Z;\mZ)_\G\to H^d(\cM)$. However, as we assume $H^+_d(X_n)$ generates the whole of $H_d(X_n)$, it is of full rank and so the positive elements in $H^d(X_n)$ will be the dual cone, of the same rank. Thus every positive will be of the form $[\alpha_n]$, as above, and this is enough to identify the positives in $C(Z;\mZ)_\G$ with those in $H^d(\cM)$.

As the $X_n$ are compact simplicial complexes, their homology is finitely generated and so the `upside down' universal coefficient theorem applies. Thus we identify $H_d(X_n)$ with its double dual
$$H_d(X_n)=\hom(H^d(X_n);\mZ)=\hom\left(\hom(H_d(X_n);\mZ);\mZ\right)\,.$$
Then  defining  positivity in $\hom(H^d(X_n);\mZ)$ by saying $F\in \hom(H^d(X_n);\mZ)$ is positive if $F(\alpha)\geq0$ for all positive $\alpha\in H^d(X_n)$, agrees with the original notion of positivity in $H_d(X_n)$. Working now with homomorphisms to $\mR$ this gives  equation (4).

Finally, equation (5) follows from the fact that when the presentation is $\mQ$-stable, the maps $L_n$ are isomorphisms.
\end{proof}

Recall that the set of $\G$-invariant probability measures on $Z$ can be identified with the convex set in $\phom_\mZ(C(Z;\mZ)_\G;\mR)$ of functionals satisfying $\int_Z\mathbf{1}_Zd\mu=1$, and the ergodic ones can be identified with the extreme points of this set. Thus, when the conditions of the theorem are met this allows us to identify the set of invariant ergodic probability measures  directly.

\section{Applications and Examples}\label{examp}

We begin by considering an example that exploits the connection between the homology core and the structure of the invariant measures of the underlying action. This first example has some overlap with the results of Cortez and Petite \cite{CortPet}.

\begin{ex}
Solenoids and $\G$ odometers
\end{ex}
Let $M$ be a $\mathrm{PL}$ orientable $d$--dimensional manifold with fundamental group $\Gamma=\pi_1(M,m_0)$. Consider then a chain of (not necessarily normal) subgroups of finite index greater than 1:
\[
 \G=\G_1 \supset \G_2\supset \G_3 \supset \cdots \supset\G_i \supset \cdots
\]
and the associated Cantor set
\[
C = \lim_{\leftarrow}  ~ \{\, \G/\G_1 \, \leftarrow \,\G/\G_2 \,
\leftarrow \,\cdots \, \leftarrow\, \G/\G_i  \, \leftarrow\cdots
\}.
\]
There is a natural minimal action of $\G$ on  $C$ given by translation in each factor. The suspension of this action over $M$ then yields a minimal matchbox manifold $\cM  = \widetilde{M}\times_{\G}C$ which has a consistent presentation in which $X_1=M$ and $X_n=\widetilde{M}/\G_n$. This presentation can be made simplicial by taking a simplicial structure for $M$ that is then lifted to $\widetilde{M}$, which in turn pushes down to a simplicial structure for the leaves and the quotients $X_n.$

\begin{thm}\label{Solenoid}
The action of $\G$ on $C$ as above is uniquely ergodic.
\end{thm}

\begin{proof}
In this case  $H_d(X_n)$ is isomorphic to $\mZ$  for each $n.$ The induced homology maps are multiplication by the degrees of the corresponding covering maps, which in turn are given by the indices of the subgroups. Thus, this presentation is $\mQ$--stable and each core $\cC_\cM(n)$ and each vector space $V_n$ can be identified with $\mR.$ By Theorem \ref{invariantsandcone}  the action of $\G$ is uniquely ergodic.
\end{proof}

We now begin an investigation of how to calculate the homology core for $\mQ$ and $\mZ$--stable presentations.

\begin{defn}
A sequence of matrices of constant rank $d$ with non--negative entries $(M_n)_{n\in \mN},$ is \emph{recurrent} if there are indices $k_1<\ell_1 \leq k_2<\ell_2 \leq \cdots$ and a matrix $B$ with positive entries satisfying for all $n$
\[
B=M_{k_n}^{\ell_{n}-1}
\]

\end{defn}

It is known, see e.g. \cite[pp. 91--95]{Furst}, that if  $(M_n)_{n\in \mN},$ is recurrent then
there is a $v\in \mR^d$ with positive entries satisfying
\[
\mathrm{span}\,v=\bigcap_{n\in \mN} \, M_0^n \left(\cC^d\right),
\]
where $\cC^d$ denotes the positive and negative cone in $\mR^d.$ Recurrent sequences have been important in the study of $S$--adic systems, see e.g. \cite{BST2014}.

It then follows that if the sequence of matrices $(M_n)_{n\in \mN}$ representing the linear maps as described in Theorem \ref{orbit} is recurrent, then $\cC_{\cM}(n)$ will be a single line for each $n.$ We shall see below in Example \ref{cf} that this condition is however not necessary for the core to be a single line in each place. In the  special case $(M_n)_{n\in \mN}$  is a sequence each term of which is the same positive matrix,   $\cC_{\cM}(n)$ is a single line formed by the span of the Perron--Frobenius right eigenvector.

\begin{ex}\label{sub}
Substitution tiling spaces
\end{ex}
We now provide two substitution tiling spaces of dimension one which can be easily distinguished topologically by their homology cores.

\begin{equation*}\label{sub1}
\begin{split}
   \sigma_1:\{a,b\}\to \{a,b\}^* \text{ is given by }& \begin{cases}
           a & \mapsto  a^{10}b^7 \\
           b & \mapsto a^3b^2
         \end{cases} \text{ and } \\
     \sigma_2:\{a,b\}\to \{a,b\}^* \text{ is given by } & \begin{cases}
           a & \mapsto  a^{11}b^4 \\
           b & \mapsto a^3b
         \end{cases} .
\end{split}
\end{equation*}
Each substitution $\sigma_i$ is primitive, aperiodic and is proper (see \cite{BDETDS} for the role properness plays in expansions as inverse limits). Thus, the corresponding tiling spaces $\cT_i$ (formed by suspending the associated substitution subshift on $\{a,b\}^\mZ$ over the circle) admit the following presentations \cite{BDETDS}

\begin{equation}\label{wedge}
 \begin{array}{lcl}
 \cT_i&\approx & \underleftarrow{\lim}\{X \xleftarrow{f_i}X \xleftarrow{f_i} X \xleftarrow{f_i} \cdots\}
\end{array}
\end{equation}

where $X$ is the wedge of two circles in both cases and the map $f_i$ is the natural one induced by the corresponding substitution $\sigma_i.$ This can be easily adjusted to yield an oriented simplicial presentation by introducing vertices in $X$ (progressively more as one passes down the sequence).
For each copy of $X$ in the two towers we take as a basis for $H_1(X)\approx \mZ \bigoplus \mZ$ the cycle corresponding to the $a \text{--circle}\approx (1,0)$ and  the cycle corresponding to the $b \text{--circle}\approx (0,1).$ Then each $V_n$ in the two sequences is isomorphic to $\mR^2$ with the corresponding bases. We then have the corresponding towers of the $V_n$ and $L_n.$

\begin{equation}\label{HomTower}
\mR^2 \xleftarrow{M_i} \mR^2 \xleftarrow{M_i} \mR^2 \xleftarrow{M_i} \cdots
\end{equation}

where $M_1=\left(
   \begin{array}{cc}
     10 & 7 \\
     3 & 2 \\
   \end{array}
 \right)
$ and $M_2=\left(
   \begin{array}{cc}
     11 & 4 \\
     3 & 1 \\
   \end{array}
 \right)
$ represent the corresponding linear transformations with respect to the chosen bases.
Observe then that both presentations are $\mZ$--stable, and so by Theorem \ref{orbit} the two tiling spaces are homeomorphic only if their homology cores are in the same $GL(2,\mZ)$ orbit. As the matrices are positive, by our above remarks the cores at all places are given by the span of the Perron-Frobenius right eigenvector of the corresponding matrix. Such eigenvectors are given by $v_1:=\left(\begin{array}{c}\frac{1}{7}(4+\sqrt{37})\\ 1\end{array}\right)$ for $M_1$ and $v_2:=\left(\begin{array}{c}\frac{1}{4}(5+\sqrt{37})\\ 1\end{array}\right)$ for $M_2$. As the continued fraction expansions of $\frac{1}{7}(4+\sqrt{37})$ and  $\frac{1}{4}(5+\sqrt{37})$ are not tail equivalent, the vectors $v_1$ and $v_2$ cannot be in the same $GL(2,\mZ)$ orbit \cite[Thm. 174]{HW}. Observe, that despite this the eigen\textit{value}  for both matrices is $6 + \sqrt{37},$ and so the distinction between these two spaces is not picked up by the invariants related to matrix equivalence or ordered cohomology \cite{BDFund},\cite{ORS},\cite{SW2003}.

\begin{ex}\label{shape}
Relation to shape type
\end{ex}
We shall see in Example \ref{cf} an entire class of shape equivalent spaces that the homology core can distinguish topologically, but here we supplement the pair $\sigma_1,\sigma_2$ with a third substitution that demonstrates a limitation of the homology core for the purposes of topological classification, showing that the core cannot distinguish all shape {\em inequivalent\/} spaces.

\begin{equation*}\label{L1}
\sigma_3:\{a,b\}\to \{a,b\}^* \text{ is given by } \begin{cases}
           a & \mapsto  ababaababaababaab \\
           b & \mapsto ababa
         \end{cases}
\end{equation*}

This substitution is not proper, but its square $(\sigma_3)^2$ is proper. The tiling space $\cT_3$  corresponding to $\sigma_3$ is the same as the tiling space corresponding to $(\sigma_3)^2$ in the sense that the subshifts of $\{a,b\}^\mZ$ determined by these substitutions are the same. Thus, again $\cT_3$  admits an oriented simplicial presentation as in Equation \ref{wedge}, where $X$ is again the wedge of two circles, but the map $f_3$ is  induced by the  substitution $(\sigma_3)^2.$  With respect to the bases as before, the homology tower for $\cT_3$ is as in Equation \ref{HomTower}  with $M_i$ replaced by $(M_1)^2.$ Thus, the homology core of  at all places is identical to that of $\cT_1.$ Observe that the bonding map $f_3$ yields an automorphism of $\pi_1(K)$ (with base point the point common to both cirles) whose inverse can be represented by the automorphism of the free group generated by $\{a,b\}$ given by the square of the following:

 \begin{equation*}\label{inv}
\begin{array}{ccl}
   a & \mapsto & a^{-1}b^3 a^{-1}b^4  \\
  b & \mapsto & b^{-4}ab^{-3}ab^{-3}a
\end{array}
\end{equation*}

It follows that the $L_1$ invariant (see \cite{ClarkHunton}) vanishes for $\cT_3.$ However, the $L_1$ invariant does not vanish for $\cT_1,$ as can be seen by an application of the folding lemma of Stallings. (See \cite{ClarkHunton} for similar examples.) Thus, although these spaces are not homeomorphic or even shape equivalent, the homology core does not detect this.

In general, once an appropriate presentation has been found as indicated in \cite{AP1998}, one can calculate the homology core of a substitution tiling space of higher dimension in a similar way using a single matrix.

We now see how one can apply Theorem \ref{orbit} to great advantage to topologically classify some natural classes of spaces that are not substitution tiling spaces but matchbox manifolds of dimension one.

For convenience to make indices match their usual interpretations, in the following two examples we will index inverse sequences starting with index 0.

\begin{ex} \label{cf}
Continued fractions
\end{ex}

For simplicity (as it does not affect the homology calculations) we represent $K$ as the $CW$ complex depicted in Figure~\ref{Simplex} with three one cells with the indicated orientations and two vertices.

\begin{figure}
\centering
\begin{tikzpicture}[->, node distance=4cm, auto]
\node [fill,circle,draw,inner sep = 0pt, outer sep = 0pt, minimum size=2mm] (at) at (0.000000,0.000000) {};
\node [fill,circle,draw,inner sep = 0pt, outer sep = 0pt, minimum size=2mm] (as) at (2.000,0.000000) {};
\draw (as) edge[bend right=110, looseness=3, ->, ultra thick] node {$a$}(at);
\draw (at) edge[ ->, ultra thick] node {$c$}(as);
\draw (as) edge[bend left=110, looseness=3, ->, ultra thick] node {$b$}(at);
\end{tikzpicture}
\caption{The complex $K$}\label{Simplex}
\end{figure}
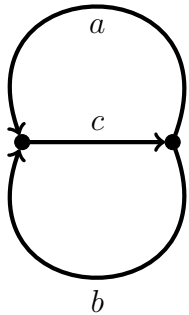

For each positive integer $n$ let $f_n:K \to K$ be the map defined by
\[
\xymatrix{
a \ar[r] &  c\\
b\ar[r]  &ca\overbrace{cb\cdots cb}^\text{$n-1$ copies}c\\
c\ar[r]&b}
\]
where $f_n$ maps each cell onto cells in the indicated order from left to right, preserving orientation.
For each sequence of positive integers $N=(n_0,n_1,...),$ we define the orientable matchbox manifold

\begin{equation}\label{cfpres}
 \begin{array}{lcl}
 X_N &:=&\underleftarrow{\lim}\;\{K \xleftarrow{f_{n_0}} K \xleftarrow{f_{n_1}} K \xleftarrow{f_{n_2}} K \xleftarrow{f_{n_3}}\cdots\}.
\end{array}
\end{equation}

For homology calculations, we use the classes $[z_1]$ and $[z_2]$ of the cycles $z_1 \sim cb$ and $z_2\sim ca$ as generators of $H_1(K,\mZ)\approx \mZ\bigoplus\mZ.$ With respect to these generators, we have the induced homomorphism on $H_1(K,\mZ)$ given by
\[
(f_n)_* \sim \left(
               \begin{array}{ll}
                 n & 1 \\
                 1 & 0 \\
               \end{array}
             \right).
\]
Notice that  $(f_n)_*$  is an isomorphism for each $n,$ and so each presentation as given in Equation \ref{cfpres} is $\mZ$--stable, and so we may apply Theorem~\ref{orbit} to the family of spaces
\[
\fX:=\{\,X_N\, | \, N \text{ is a sequence of positive integers }\}.
\]
Observe that
\[
\left(\begin{array}{ll}n_0 & 1 \\1 & 0 \\\end{array}\right)   \left(\begin{array}{ll}n_1 & 1 \\1 & 0 \\\end{array}\right)\cdots \left(\begin{array}{ll}n_k & 1 \\1 & 0 \\\end{array}\right)=\left(\begin{array}{ll}p_{k} & p_{k-1} \\q_{k} & q_{k-1} \\\end{array}\right)
\]
where $\frac{p_k}{q_k}=[n_0,n_1,n_2,\dots,n_k]$ in continued fraction notation. Observe that with $\alpha_N:=[n_0,n_1,n_2,\dots],$ we have that $\lim_{k\to \infty}\frac{p_k}{q_k}=\alpha_N.$ Now $\left(\begin{array}{ll}p_{k} & p_{k-1} \\q_{k} & q_{k-1} \\\end{array}\right)$ maps the  positive and negative cones in $V_k$ to the sectors in $V_0$ bounded by the lines spanned by  $\left(\begin{array}{r}p_k \\ q_k\end{array}\right)$ and $\left(\begin{array}{r}p_{k-1} \\ q_{k-1}\end{array}\right).$ Hence, we have that the homology core of $X_N$ at place zero is given by
\[
\mathrm{span} \left(\begin{array}{r}\alpha_N \\ 1\end{array}\right) = \bigcap_{n\in \mN} \, M_0^n \left(\cC_n \right).
\]
Hence, the corresponding $\mZ$ action on the Cantor set is uniquely ergodic by Theorem \ref{invariantsandcone}, and by Theorem \ref{orbit} the spaces $X_N$ and $X_M$ are homeomorphic only if there is a matrix in $GL(2,\mZ)$ that maps $\left(\begin{array}{r}\alpha_N \\ 1\end{array}\right)$ into $\mathrm{span}\,\left(\begin{array}{r}\alpha_M \\ 1\end{array}\right).$ By the classical theorem on the classification of numbers by their continued fraction expansions, \cite[Thm. 174]{HW}, this happens precisely when the continued fraction expansions for $\alpha_N $ and $\alpha_M$ share a common tail: there exist $k$ and $l$ such that for all positive integers $i$ we have $m_{k+i}=n_{\ell + i}.$ When this happens, the inverse sequences defining  $X_N$ and $X_M$ have equal cofinal subsequences and so are clearly homeomorphic. Thus we obtain the following classification of the spaces in $\fX,$ c.f. \cite{BW},\cite{F}.

\begin{thm}
$X_M$ and $X_N$ are homeomorphic if and only if the sequences $M$ and $N$ share a common tail.
\end{thm}

Observe that \textit{all} the spaces in $\fX$ are shape equivalent to $K$ and hence to a wedge of two circles. Hence, while the Example \ref{shape} illustrates that there are spaces for which shape invariants (such as the $L_1$ invariant) can distinguish spaces that are not distinguished by their homology cores, there are also large classes of shape equivalent spaces that the homology core can distinguish.

Consider the following three periodic sequences $N_1=(\overline{12}),N_2=(\overline{1,2,3})$ and $N_3=(\overline{2,1,3}).$ Letting $g_1=f_{12},$ $g_2=f_1\circ f_2 \circ f_3$ and $g_2=f_2\circ f_1 \circ f_3$, the spaces $X_{N_i}$ are homeomorphic to the spaces  $\underleftarrow{\lim}\;\{K \xleftarrow{g_i} K \xleftarrow{g_i} K  \xleftarrow{g_i}\cdots\}$. The map induced on homology by the $g_i$ is given by the three matrices $M_1:=\left(
               \begin{array}{rr}
                 12 & 1 \\
                 1 & 0 \\
               \end{array}
             \right), M_2:=\left(\begin{array}{ll}1 & 1 \\1 & 0 \\\end{array}\right)   \left(\begin{array}{ll}2 & 1 \\1 & 0 \\\end{array}\right) \left(\begin{array}{ll}3 & 1 \\1 & 0 \\\end{array}\right)$ and
$M_3:=\left(\begin{array}{ll}2 & 1 \\1 & 0 \\\end{array}\right)   \left(\begin{array}{ll}1 & 1 \\1 & 0 \\\end{array}\right) \left(\begin{array}{ll}3 & 1 \\1 & 0 \\\end{array}\right).$ Now, each $M_i$ has the same characteristic equation and therefore the same eigenvalues $6 \pm \sqrt{37}$. The larger eigenvalue represents the expansion factor for these matrices. However, the eigenvectors for these three matrices that correspond to $6 + \sqrt{37}$ are given by   $\left(\begin{array}{r}\alpha_{N_i} \\ 1\end{array}\right),$ which are pairwise inequivalent since the continued fractions of the $\alpha_{N_i}$ do not share common tails. (It also follows that no pair of the matrices $M_i$ is conjugate in $GL(2,\mZ)$, for otherwise the corresponding pair of eigenvectors would be in the same in $GL(2,\mZ)$ orbit.) Thus, we see again that the homology core is capturing more than just the information given by the expansion factor \cite{ORS}.

Similar examples of families with matrices of larger size can be obtained using the matrices corresponding to higher dimensional versions of continued fractions, see, e.g., \cite{HK}.

\begin{ex}\label{gcf}
 Generalised continued fractions
\end{ex}

With $K$ as before, for each pair of positive integers $m,n$ let $ f_{m,n}:K \to K$ be the map defined by

\[
\begin{array}{ll}
a \rightarrow &  c\\
b\rightarrow  &  \overbrace{ca\cdots ca}^\text{$m$ copies} \overbrace{cb\cdots cb}^\text{$n-1$ copies}c\\
c\rightarrow & b
\end{array}
\]

With respect to the same generators as before, we have the induced isomorphism on $H_1(K,\mQ)$ given by

\[
(f_{m,n})_* \sim \left(
               \begin{array}{ll}
                 n & 1 \\
                 m & 0 \\
               \end{array}
             \right).
\]
Thus while our presentation of $ X(\alpha,\beta)$ is not generally $\mZ$--stable, it is $\mQ$--stable. For given sequences of positive integers $\alpha =(a_1,a_2,\dots),\beta =(b_0,b_1,\dots)$ we will then have the orientable matchbox manifold given by
$$\begin{array}{lcl}
 X(\alpha,\beta) &:=&\underleftarrow{\lim}\;\{K \xleftarrow{f_{1,b_0}} K \xleftarrow{f_{a_1,b_1}} K \xleftarrow{f_{a_2,b_2}} K \xleftarrow{f_{a_3,b_3}}\cdots\}.
\end{array}$$
For given sequences $\alpha$ and $\beta$ we then have
\[
\left(\begin{array}{ll}b_0 & 1 \\1 & 0 \\\end{array}\right)   \left(\begin{array}{ll}b_1 & 1 \\a_1 & 0 \\\end{array}\right)\cdots \left(\begin{array}{ll}b_k & 1 \\a_k & 0 \\\end{array}\right)=\left(\begin{array}{ll}A_{k} & A_{k-1} \\B_{k} & B_{k-1} \\\end{array}\right)
\]
where  $\frac{A_k}{B_k}$ is the $k$--th convergent of the generalised continued fraction
$$b_0 + \mathrm{\mathbf{K}}_{n=1}^{\infty}\frac{a_n}{b_n},$$ see, e.g., \cite{JT}.

We can rewrite our general continued fraction with an equivalent one (one with the same convergents) as follows \cite[2.3.24]{JT}

\[
  b_0+ \cfrac{a_1}{b_1
          + \cfrac{a_2}{b_2
          + \cfrac{a_3}{b_3 + \cfrac{a_4}{b_4+\cdots} } } }\approx b_0 + \cfrac{1}{b_1/a_1
          + \cfrac{1}{a_1b_2/a_2
          + \cfrac{1}{b_3a_2/a_1a_3 + \cfrac{1}{a_1a_2b_4/a_2a_4+\cdots} } } }
\]
where now we have positive rational entries. By the theorem of Van Vleck
 \cite[Thm. 4.29]{JT}, for a continued fraction with positive entries of the form  $$b_0 + \mathrm{\mathbf{K}}_{n=1}^{\infty}\frac{1}{b_n},$$
we have that if $\sum_{i=1}^\infty b_i$ converges, then the even and odd convergents converge monotonically to  different values (the larger/smaller terms decrease/increase), and if $\sum_{i=1}^\infty b_i$ diverges, then the convergents converge to a single value.

This leads to the following conclusions. In what follows, we use the notation $k_1=b_1/a_1$ and recursively $k_n=\frac{b_nb_{n-1}}{a_nk_{n-1}}$ for $n>1.$

\begin{prop}\label{gencf}
For $X(\alpha,\beta)$

\begin{enumerate}
  \item If $\sum_n^\infty k_n$ diverges, then with $\ell=\lim \frac{A_n}{B_n}$, we have the homology core at place zero given by
  \[
\mathrm{span} \left(\begin{array}{r}\ell \\ 1\end{array}\right) = \bigcap_{n\in \mN} \, M_0^n \left(\cC_n\right).
\]
  \item  If $\sum_n^\infty k_n$ converges, then with $\ell_E=\lim \frac{A_{2n}}{B_{2n}}$ and $\ell_O=\lim \frac{A_{2n-1}}{B_{2n-1}}$ we have the homology core at place zero  $\bigcap_{n\in \mN} \, M_0^n \left(\cC_n\right)$ is the union of the two sectors in $V_0$ bounded by $\mathrm{span} \left(\begin{array}{r}\ell_E \\ 1\end{array}\right)$ and $\mathrm{span} \left(\begin{array}{r}\ell_O \\ 1\end{array}\right).$
\end{enumerate}
\end{prop}
Observe that two spaces $X(\alpha,\beta)$ and $X(\alpha',\beta')$ are homeomorphic only if they both correspond to the same case (1) or (2) as above. If both spaces satisfy the conditions for (1), the spaces are homeomorphic only if the  vectors $\left(\begin{array}{r}\ell \\ 1\end{array}\right)$  and $\left(\begin{array}{r}\ell' \\ 1\end{array}\right)$ are in the same $GL(2,\mQ)$ orbit and the corresponding $\mZ$ actions are uniquely ergodic as follows from Theorem \ref{invariantsandcone}. However, if both spaces satisfy the conditions for (2), the spaces are homeomorphic only if the vectors $\left(\begin{array}{r}\ell_E \\ 1\end{array}\right)$ and $\left(\begin{array}{r}\ell_O \\ 1\end{array}\right)$ are (as a pair) in the same $GL(2,\mQ)$ orbit as $\left(\begin{array}{r}\ell'_E \\ 1\end{array}\right)$ and $ \left(\begin{array}{r}\ell'_O \\ 1\end{array}\right).$  We see then that in this case the corresponding $\mZ$ actions have two invariant ergodic probability measures by Theorem \ref{invariantsandcone}.

Besides recurrence, a simple criterion that guarantees that we are in case (1) is given by $b_n>a_n$ for sufficiently large $n$, and a simple example of case (2) is given by $\alpha=(2^{2n-1})_{n\in\mZ^+},\beta=(1,1,1,\dots).$

\section*{Acknowledgments}
We thank Ian Leary for helpful comments concerning duality for manifolds, in particular for the approach using equivariant cohomology deployed in the proof of Proposition \ref{cohMM}. We thank Steve Hurder for helpful communications regarding the structure of the invariant measures of foliated spaces. We also thank the Leverhulme Trust for its generous support through the grant  IN-2013-045 that aided the collaboration that made this research possible.


\end{document}